\documentclass[a4paper]{amsart}

\usepackage[utf8]{inputenc}
\usepackage{amssymb,pstricks,amscd,mathrsfs}
\usepackage{bbm}
\usepackage[T1]{fontenc}
\usepackage[totalwidth=15cm,totalheight=22cm]{geometry}
\usepackage{esint}
\usepackage{enumerate}
\usepackage[shortlabels]{enumitem}
\usepackage[colorlinks=true,urlcolor=blue,citecolor=red,linkcolor=blue,linktocpage,pdfpagelabels,bookmarksnumbered,bookmarksopen]{hyperref}
\usepackage{xcolor}
\usepackage{extarrows}

\newtheorem{cor}{Corollary}[section]
\newtheorem{thm}[cor]{Theorem}
\newtheorem{prop}[cor]{Proposition}
\newtheorem{lemma}[cor]{Lemma}

\theoremstyle{definition}
\newtheorem{defi}[cor]{Definition}
\theoremstyle{remark}
\newtheorem{remark}[cor]{Remark}

\newcommand{\N}{{\mathbb N}}
\newcommand{\Q}{{\mathbb Q}}

\newcommand{\R}{{\mathbb R}}

\renewcommand{\S}{{\mathbb S}^m}
\newcommand{\Ri}{\mathbb{R}^+ \cup \{+\infty\}}

\newcommand{\B}{\mathcal B}
\newcommand{\E}{\mathfrak{E}}
\newcommand{\G}{\mathfrak{G}}
\newcommand{\K}{\mathcal K}

\def\res{\mathop{\hbox{\vrule height 7pt width .5pt depth 0pt \vrule height .5pt width 6pt depth 0pt}}\nolimits}

\def\Ga{\Gamma (\lambda,\mu)}
\newcommand{\be}{\begin{eqnarray}}
\newcommand{\ee}{\end{eqnarray}}

\renewcommand{\vec}[1]{\overrightarrow{#1}}
\renewcommand{\tilde}[1]{\widetilde{#1}}
\renewcommand{\bar}[1]{\overline{#1}}
\renewcommand{\hat}[1]{\widehat{#1}}
\newcommand{\ep}{\varepsilon}
\newcommand{\supp}{\mbox{supp}\,}

\newcommand{\di}{\mathrm{Diam \,}}
\renewcommand{\phi}{\varphi}

\begin{document}

\title{On the Gauss image problem}

\author{J\'er\^ome Bertrand}
\address{Institut de Math\'ematiques de Toulouse, UMR CNRS 5219 \\
Universit\'e Toulouse III \\
31062 Toulouse cedex 9, France}
\email{bertrand@math.univ-toulouse.fr}
\date{}
\thanks{\hspace*{0.5 cm}{\it Keywords: }Convex bodies, Optimal mass transport, Gauss image problem.}

\begin{abstract}
In this note, we solve the Gauss image problem given two Borel measures on the unit sphere, one of which is absolutely continuous with respect to the uniform measure. 

\end{abstract}
\maketitle




\section{Introduction}

In this note, we consider the Gauss image problem. Namely, given two (sub)-measures $\lambda$ and $\mu$ on $\S$, does there exist a convex body $K \subset \R^{m+1}$ with the origin in its  interior such that, viewed as a function on the Lebesgue $\sigma$-algebra,
\begin{equation}\label{eq-GIPro}
\mu = \lambda \circ (\mathcal{G}_K \circ \vec{\rho}_K),
\end{equation}
where $\mathcal{G}_K$ is  the Gauss (multi-)valued map and $\vec{\rho}_K: \S \longrightarrow \partial K$ is the radial map  of the convex body $K$. Let us recall that for $k \in \partial K$, $\mathcal{G}_K(k)$ is the set of unit outward normal vectors to $\partial K$ at $k$.

When $\lambda$ is  the uniform probability measure on $\S$, this problem is known as Aleksandrov's curvature prescription problem. It has been solved by Aleksandrov himself and was (re)-proved by various means since then. One of them is based on the theory of optimal mass transport \cite{Oliker, Bertrand-GeomD}; this is also the method used in this note.

\smallskip

The following condition is known as Aleksandrov's condition:

\begin{defi}[Aleksandrov's condition]\label{def-alex-cond}
Two probability measures $\lambda$ and $\mu$ are said to be Aleksandrov related if for any  closed spherically convex subset $\omega \subsetneq \S$:

$$\lambda (\omega^*) + \mu(\omega) < 1,$$
where $\omega^*=\{x \in \S; \forall w \in \omega \; \langle x, w\rangle \leq 0\}$.
\end{defi}

Aleksandrov proved that when $\lambda$ is the uniform probability measure, the problem \eqref{eq-GIPro} can be solved if and only if $\lambda$ and $\mu$ are Aleksandrov related, we refer to his book \cite{AlexCP} for more on this.

In 2020, Böröczky {\it et al} solved the Gauss image problem for probability measures satisfying Aleksandrov's condition and such that $\lambda$ is absolutely continuous with respect to the uniform measure on $\S$ \cite{Boroczky-20}. See also \cite[Theorem 1.7]{Bertrand-GeomD} for an earlier proof in disguise of this result. Several noticeable measures in convex geometry satisfy relations like \eqref{eq-GIPro} for well-chosen absolutely continuous measures on $\S$, this is for instance the case of the \emph{dual curvature measures}, we refer to \cite{huang} for more on these measures.

In this note we shall solve the Gauss image problem under a weaker assumption than Aleksandrov's one.

\begin{defi}\label{def-WAA}
We say that the probability measures $\mu$ and $\lambda$ satisfy the \emph{weak Aleksandrov condition} if there exists $\alpha \in (0,\pi/2)$ such that for any closed set $F$ contained in a closed hemisphere of $\S$ the following inequality holds:
\begin{equation}\label{eq-WAA}
\mu(F) \leq \lambda(F_{\pi/2-2\alpha}), 
\end{equation}
where $F_{\pi/2-2\alpha} =\{n \in \S; \exists\, x \in F, \langle n , x \rangle > \cos (\pi/2 -2\alpha) \}$. 
\end{defi}

\begin{remark}
For a closed spherically convex subset  $\omega \subsetneq \S$, the set $\omega^*$ can also be defined by the equality $ \S \setminus \omega^*= \omega_{\pi/2}$, see for instance \cite{Bertrand-GeomD}. Aleksandrov's condition can thus be rephrased as 
$$ \mu (\omega) < \lambda(\omega_{\pi/2}).$$ 

Note that Aleksandrov's condition prevents the measure $\mu$ from being supported on a closed hemisphere.

\end{remark}

The fact that Aleksandrov's condition implies \eqref{eq-WAA} was proved in \cite[Proposition 3.5]{Bertrand-GeomD}. Moreover, the weak Alexandrov condition is the sharp condition for the Gauss image problem:

\begin{remark} Let $K\subset \R^{m+1}$ be a compact convex set with the origin in its interior then the angle between a point in the boundary of $K$ and an outward normal vector to $K$ at that point is uniformly bounded above by a constant $\pi/2-\alpha<\pi/2$. In other terms, \eqref{eq-WAA} is a necessary condition for the Gauss image problem \eqref{eq-GIPro} to have a solution, see \eqref{eq-angle-bound} for more details. 

On the contrary, Aleksandrov's condition is \emph{not} necessary when the support of $\lambda$ is not the whole sphere \cite{Semenov-exi}.

\end{remark}

\smallskip

The main result of this note is  the following theorem.

\begin{thm}\label{main} Let $\lambda$ and $\mu$ be two probability measures on the unit sphere $\S$ and assume that $\lambda$ is absolutely continuous with respect to the uniform measure and $\mu$ is not concentrated on a closed hemisphere. Then the Gauss image problem for $\lambda$ and $\mu$ admits a solution provided that $\lambda$ and $\mu$ satisfy the weak Aleksandrov condition. 

Assume there are two convex bodies $K,L \subset \R^{m+1}$ with the origin in their interiors and solutions to \eqref{eq-GIPro}. Then, the following holds for any Borel set $\omega$
\begin{equation}\label{eq-fo-uni}
 \lambda ((\mathcal{G}_k\circ \vec{\rho}_K(\omega)) \Delta \, (\mathcal{G}_L\circ \vec{\rho}_L(\omega)))=0,
\end{equation}
where $A \Delta B$ is the symmetric difference of the sets $A$ and $B$.
\end{thm}

\begin{remark}
We are currently not interested in submeasures so we don't know whether our proof can be adapted to this larger framework. Moreover, to our knowledge, the interesting examples are all measures. 
\end{remark}

The second statement \eqref{eq-fo-uni} is a result of uniqueness for multivalued map up to a $\lambda$-negligible set.

\smallskip

In late 2022, V. Semenov put on Arxiv a paper in which he solves the Gauss image problem for $\lambda$ an absolutely continuous measure and $\mu$ a measure with finite support (which corresponds to the case where the underlying convex body is a convex \emph{polyhedron}) \cite{Semenov-exi}. The uniqueness property stated in Theorem \ref{main} is also proved by another method in \cite{Semenov-uni}; our proof is shorter. Our proofs are not based on the aformentioned articles, instead our approach can be seen as a generalisation of the method introduced in our earlier work \cite{Bertrand-GeomD}.

Building on \eqref{eq-fo-uni}, Semenov proves the following corollary.

\begin{cor}[\cite{Semenov-uni}] Under the assumptions of Theorem \ref{main}, the solution to the Gauss image problem is uniquely determined \emph{up to dilation} on each \emph{rectifiable} path connected component of $\supp \lambda$.
\end{cor}

One could expect the solution of the Gauss image problem to be uniquely determined on each \emph{connected} component of $\supp \lambda$; this question is related to fine properties of compact Euclidean sets supporting an absolutely continuous measure and might depend on the dimension of the space. To our knowledge, this is an open question.

\smallskip

Compared to our earlier work \cite{Bertrand-GeomD}, the main new ingredient is a new method to construct a \emph{Kantorovitch potential}, a classical object in the theory of optimal mass transport that happens to define a convex body when the cost function is chosen properly. This construction seems to be the first one where no connectedness of the support of the measure is required (actually what really matters in the previous works is the fact that all the elements of the support are contained in the same equivalence class relative to a relation defined in terms of the cost function: we remove this restriction), compare to \cite{Bertrand-GeomD, BP, BPP, arstein}. This improved construction is expected to be useful in other contexts.

In the next section, we recall some classical notions in convex geometry and introduce related notation. The third part provides a quick introduction to optimal mass transport together with explanations regarding the connections between the Gauss image problem and optimal mass transport. In Section 4, we introduce some suitable discretisations of the measures involved before studying combinatorial properties of graphs related to these discretisations. Finally, in the last part we construct the Kantorovitch potential then prove \eqref{eq-fo-uni}; these properties complete the proof of our main theorem.



\section{Preliminaries}

We assume the reader is familiar with basic results in convex geometry, we refer to the textbooks \cite{Rocka70, schneider} for a detailed introduction. We denote by $C(\S)$  the set of continuous functions defined on $\S$ while we write $C^+(\S)$ for the set of positive continuous functions. To a function $\rho \in C^+(\S)$ corresponds a star-shaped compact set $D$ with respect to the origin. The set $D$ is defined by the property that its boundary $\partial D$ is biLipschitz homeomorphic to $\S$ through the \emph{radial map}

$$\vec{\rho} := \left\{\begin{array}{lcl}
\S & \longrightarrow & \partial D \\
x & \longmapsto & \rho(x)\, x
\end{array}\right.
$$

The function $\rho$ is called \emph{the radial function} of $D$, we also use the notation $\rho_D$ to emphasize the relation with the set $D$. Its definition can be extented to $\R^{m+1}$ as $x \longmapsto 1/|x| \rho(x)$ where $|\cdot|$ stands for the Euclidean norm on $\R^{m+1}$.
The set of all such star-shaped compact sets is denoted by $\mathcal{E}$. 

The \emph{support function} of $D \in \mathcal{E}$ is the function $h_D: \R^{m+1} \longrightarrow (0,+\infty)$ defined by 
\begin{eqnarray}\label{eq-def-support}
h_D(n) &= & \max \{ \langle x, n\rangle; x \in D\} \nonumber \\
     &=& \max  \{ \rho_D(x) \langle x,n\rangle; x \in \S\}
\end{eqnarray}

We recall the definition of the \emph{polar set} $D^{\circ}$ of $D$ as
$$D^{\circ} = \{n \in \R^{m+1}; \forall x \in D, \langle n,x\rangle \leq 1\}= \{n \in \R^{m+1}; \forall x \in \S, \rho_D(x)\langle n,x\rangle \leq 1\}.$$

By definition of the polar transform $D^{\circ\circ}:= (D^{\circ})^{\circ}$ always contain $D$. Moreover, the equality $D^{\circ\circ}=D$ holds if and only if $D$ is a compact convex set that contains the origin in its interior. In particular, if we let $\K_0$ be the set of compact convex sets with the origin $0 \in \R^{m+1}$ in their interior, the polar transform is an \emph{involution} from $\K_0$  to itself. We shall simply call \emph{convex body} an element of $\K_0$. Moreover, the inclusion $D^{\circ\circ} \supset D$ can be rephrased as $\rho_{D^{\circ\circ}} \geq \rho_D$ and the last equality holds as functions iff $D$ is a convex body.

For $K\in \K_0$, let $\varepsilon>0$ be such that the open Euclidean ball of radius $\varepsilon$ centered at the origin is contained in $K$. Then, for $n,x \in \S$ such that $h_K(n)= \rho_K(x) \langle n,x\rangle $, we infer
\begin{equation}\label{eq-angle-bound}
 \langle n,x\rangle =\frac{h_K(n)}{\rho_K(x)} > \frac{\varepsilon}{\max \rho_K} =:\varepsilon' >0
\end{equation}
and thus \eqref{eq-WAA} holds for any closed set where $\alpha= 1/2(\pi/2 - \arccos (\varepsilon'))$.

\smallskip

Another important connection between the polar transform of $D \in \mathcal{E}$ and the radial and support functions is given by the relation
$$\rho_{D^{\circ}} = \frac{1}{h_D}.$$

It can be used to rephrase the equality $\rho_{D^{\circ\circ}} = \rho_D$ in terms of $\rho_D$ and $h_D$ only:
\begin{equation}\label{eq-radial} 
\rho_D(x) =\rho_{D^{\circ \circ}}(x)=  \frac{1}{h_{D^{\circ}}}(x) =\frac{1}{\sup_{n \in \S}\frac{\langle x,n\rangle}{h_D(n)}}=\inf_{n \in \S; \langle x,n\rangle >0} \left\{\frac{h_D(n)}{\langle x,n\rangle}\right\}.  
\end{equation}

The expression of $\rho_D$ in terms of $h_D$ is somehow similar without being identical to \eqref{eq-def-support}. Oliker's change of functions \cite{Oliker} allows one to recover a symmetrical and additive expression of the new functions.

\begin{defi}[Oliker's change of functions]
For $D \in \mathcal{E}$, set

$$\phi (n) =-\ln (h_D(n)) \mbox{   and   }  \psi(x)= \ln(\rho_D(x)).$$

Then $D$ is a convex body if and only if

\begin{equation}\label{eq-rad-supp-rela}\left\{ \begin{array}{lcl}
\phi(n) &=& \min c(n,x) - \psi (x) \\
\psi(x)&=& \min c(n,x) -\phi(n),
\end{array}\right.
\end{equation}

where the function $c: \S \times \S \longrightarrow [0,+\infty]$ is defined by $c(n,x) =-\ln  \langle n, x \rangle   \mbox{ if } \langle n,x\rangle  >0$ and $+\infty$ otherwise.

\end{defi}

This follows from simple computations that can be found in \cite{Oliker,Bertrand-GeomD}. The above relations between $\phi$ and $\psi$ are very classical in the theory of optimal mass transport. These links between the Gauss image problem and optimal mass transport are to be described in the next part.

\section{Optimal mass transport}\label{CP&OMT}


In this part, we  briefly describe the optimal mass transport problem on $\S$ and introduce related notation. For more on the subject, we refer to \cite{Vi03,Vi08}. This problem involves two Lebesgue probability measures $\mu,\lambda $ on $\S$, and a cost function $c: \S\times \S \rightarrow \Ri $. We also need to introduce the set of \emph{transport plans} $\Gamma(\lambda,\mu)$, namely the set of probability measures $\pi \in \mathcal{P}(\S\times \S)$ such that for any Lebesgue set $A\subset \S$
\begin{equation}\label{planmargi} \lambda(A) = \pi(A\times \S) \mbox{  and  } \mu(A)= \pi(\S\times A).
\end{equation}

Let us recall the cost function $c$ we consider
\begin{equation}\label{costdefi}
 c(n,x) = \left\{\begin{array}{lr}
-\log \langle n, x \rangle = -\log \cos d(n,x) & \mbox{ if } d(n,x) < \pi/2 \\
+ \infty & \mbox{ otherwise, }
\end{array}\right.
\end{equation}
where $d(n,x)$ stands for the spherical distance between $n$ and $x$.

The cost function $c$ satisfies a standard set of assumptions in the field  with the noticeable exception that it is \emph{not} real-valued. Therefore, some standard results do not apply to $c$. We gather the easy-to-prove properties of the cost function in the lemma below.  
\begin{lemma} The cost function $c: \S \times \S \longrightarrow \Ri $ defined in (\ref{costdefi}) is a continuous map. Moreover, restricted to the open set $\{c < + \infty\}$, the function $c$ is a strictly convex and increasing smooth function of the spherical distance. Consequently, for $(n,x)$ in any fixed open set $\Omega$ such that  
$\bar{\Omega} \subset \{c < + \infty\}$, the function $(n,x) \mapsto c(n,x)$ is a Lipschitz differentiable function on $\Omega$.
\end{lemma} 

The mass transport problem consists in studying

\begin{equation}\label{massdefi}
 \inf_{\pi \in \Ga} \int_{\S \times \S} c(n,x)\, d\pi(n,x).
 \end{equation}

It is customary in the field to assume that the mass transport problem is well-posed, namely that the infimum in \eqref{massdefi} is finite. The well-posedness of the problem is not an immediate consequence of the weak Aleksandrov condition and requires to be proved. 
The scheme of proof of the well-posedness is identical to that in our paper \cite{Bertrand-GeomD}, indeed only the weak Aleksandrov condition is required in that part. Details are given in the appendix.

Equipped with the topology induced by the weak convergence of probability measures, the set $\Ga$ is a \emph{compact set} as a consequence of the Banach-Alaoglu theorem. Therefore by combining this compactness property together with the continuity of the cost function, we infer the existence of minimizers in the problem above whenever the infimum is finite. These minimizers are called \emph{optimal} transport plans.

In order to tackle the Gauss image problem we shall focus our attention on a \emph{dual problem} to the mass transport problem introduced by Kantorovitch.

Let us define ${\mathcal A}$ as the set of pairs $(\phi,\psi)$ of Lipschitz functions defined on $\S$ that satisfy
\begin{equation}\label{eq-def-A}
\phi(n)+\psi(x) \leq c(n,x) \mbox{ for all }x,n \in \S.
\end{equation}
Kantorovitch's variational problem consists in studying
\begin{equation}\label{massdefi2}
\sup_{(\phi,\psi) \in \mathcal{A}}  \left\{ \int_{\S} \phi(n) d\lambda(n) + \int_{\S} \psi(x) d\mu(x)\right\}. 
\end{equation}

It is easy to see that the quantity above is always smaller or equal to (\ref{massdefi}). Indeed, given $(\phi,\psi) \in \mathcal{A}$ and $\pi \in \Ga$, we have

\begin{equation}\label{eq-Kanto-ineq}
\int_{\S} \phi(n) \,d\lambda(n)+ \int_{\S}\psi(x) \, d\mu(x)= \int_{\S\times \S} (\phi(n)+\psi(x)) \,d\pi(n,x) \leq \int_{\S\times \S} c(n,x) \, d\pi(n,x).
\end{equation} 

It can be proved that $ (\ref{massdefi}) =(\ref{massdefi2})$ whenever the cost function is continuous and nonnegative; this type of result is called \emph{Kantorovitch's duality}.   However, solution to Kantorovitch's dual problem may not exist for the cost \eqref{costdefi} because this function is not real-valued. We refer to \cite{Vi08} for more on this.

Our main goal is to prove the existence of a solution to the dual problem under the hypotheses of Theorem \ref{main}. In order to explain our method, let us first recall the definition of the \emph{$c$-transform} of a function $f \in C(\S)$ (the $c$-transform is well-defined whenever $f$ is bounded above): 
$$f^c(n):= \inf_{x \in \S}\, c(n,x) - f(x).$$
Since the cost function is symmetric, the same definition applies to functions of the variable $n$. 

\smallskip

Given that the functional

$$(\phi, \psi) \longmapsto   \int_{\S} \phi(n) \,d\lambda(n)+ \int_{\S}\psi(x) \, d\mu(x)$$

is nondecreasing with respect to each variable, and the constraint \eqref{eq-def-A}, it is not surprising to seek for maximisers of Kantorovitch's problem among pairs of functions of the form $(\phi, \phi^c)$ where $\phi^{cc} := (\phi^c) ^c$ coincides with the function $\phi$. Such a function $\phi$ is called a \emph{Kantorovitch potential}, besides $h(n) := e^{-\phi(n)}$ and $\rho(x):= e^{\phi^c(x)}$ satisfy \eqref{eq-radial} and thus determine a unique convex body.

Given a Kantorovitch potential $\phi$, the  \emph{$c$-subdifferential of $\phi$}, whose definition is recalled below, is a useful set in connection with the mass transport problem:
\begin{equation}\label{def-csubdif}
\partial_c \phi := \{(n,x)\in \S\times \S; \phi(n) + \phi^c(x)=c(n,x)\}.
\end{equation}

The $c$-subdifferential of a function is a particular instance of a \emph{$c$-cyclical monotone set} $S \subset \S\times \S$ whose definition is
\begin{equation}\label{eq-def-cmono}
\sum_{1\leq i \leq k} c(n_{i+1},x_i) \geq \sum_{1\leq i \leq k} c(n_{i},x_i),
\end{equation} 
where $k$ is any positive integer, $(n_1,x_1),\cdots, (n_k,x_k) \in S$ are arbitrary, and $n_{k+1}=n_1$.

This condition is equivalent to optimality for mass transport problem involving finitely supported measures. More in general, the support of an \emph{optimal} transport plan in our set-up has to be a $c$-cyclically monotone set. We refer to \cite[Chapter 5]{Vi08} for the proof of this property and more.
 
\smallskip

Our main result regarding Kantorovitch's dual problem is

\begin{thm}\label{KantoPotenI}Let $\lambda$ and $\mu$ two probability measures on $\S$  satisfying the weak Aleksandrov condition $\eqref{def-WAA}$. Assume that $\lambda$ is absolutely continuous while the support of $\mu$ is not contained in a closed hemisphere. Then, denoting by $\pi_o$ an optimal plan in $\Gamma (\lambda,\mu)$, there exist a Kantorovitch potential $\varphi$ such that 
\begin{equation}\label{thm-gamm-ccon}
 \supp \pi_o \cap \{c<+\infty\} \; \subset \partial_c \varphi.
\end{equation}
\end{thm}

An easy consequence of this result -if we discard for now the question of the regularity of  $\phi$ and $\phi^c$-  is that equality holds in \eqref{eq-Kanto-ineq}:

$$ \int_{\S} \phi(n) \,d\lambda(n)+ \int_{\S}\phi^c(x) \, d\mu(x)= \int_{\S\times \S} (\phi(n)+\phi^c(x)) \,d\pi_o(n,x) \stackrel{\eqref{thm-gamm-ccon}}{=} \int_{\S\times \S} c(n,x) \, d\pi_o(n,x).$$

The above equality forces the measure $\mu$ to be the Gauss image measure of the underlying convex body $K \in \K_0$ determined by the pair $(\phi, \phi^c)$. Indeed, using Oliker's change of functions  $(\phi,\phi^c) \leftrightarrow (\rho_K,h_K)$, one infer the equality
$$\partial_c \phi = \{(n,x); h_K(n) =\rho_K(x) \langle n, x \rangle\}= \{(n,x); n \in \mathcal G_K\circ \vec{\rho}_K(x)\}.$$

For the sake of completeness, we reproduce the short proof of the fact that $\mu$ is the Gauss image of $K$ relative to $\lambda$ \cite{Bertrand-GeomD}.



For any Borel set $U \subset \S$, it holds
\begin{eqnarray*}
\mu(U) &=& \pi_o (\S\times U)\\
             &=& \pi_o(\S \times U\cap \{(n,x) \in (\S)^2; \phi(n) +\phi^c(x)= c(n,x) \})\\
             &=& \pi_o(\S \times U\cap \{(n,x) \in (\S)^2; n \in \mathcal{G}_K (\vec{\rho}_K(x))\})\\
             &=& \pi_o(  \mathcal{G}_K \circ \vec{\rho}_K (U) \times U \cap  \{(n,x); n \in  \mathcal{G}_K(\vec{\rho}_K(x))\})\\
             &=& \pi_o(  \mathcal{G}_K\circ \vec{\rho}_K(U) \times \S  \cap   \{(n,x); n \in  \mathcal{G}_K(\vec{\rho}_K(x))\})\\
            &=&  \pi_o( \mathcal{G}_K\circ \vec{\rho}_K(U) \times \S ) \\
             &=& \lambda ( \mathcal{G}_K\circ \vec{\rho}_K(U))
\end{eqnarray*}

where to get the equality in line 6 we use:
\begin{multline*}
   \mathcal{G}_K\circ \vec{\rho}_K(U) \times  U^c\, \cap  \{(n,x) \in (\S)^2; n \in \mathcal{G}_K(\vec{\rho}_K(x)))\} \subset \\
 \{(n,x)  \in (\S)^2 ; \exists\, x'\neq x, n \in  \mathcal{G}_K(\vec{\rho}_K(x))\cap   \mathcal{G}_K(\vec{\rho}_K(x'))\}
 \end{multline*}
 which yields
 \begin{multline*}
 \pi_o( \mathcal{G}_K\circ \vec{\rho}_K(U) \times  U^c \cap  \{(n,x) \in (\S)^2; n \in \mathcal{G}_K(\vec{\rho}_K(x)))\}) \leq \\
 \lambda( \{n  \in \S ; \exists\, x'\neq x, n \in  \mathcal{G}_K(\vec{\rho}_K(x))\cap  \mathcal{G}_K(\vec{\rho}_K(x'))\})=0,
 \end{multline*}
where the last equality hods since $\lambda$ is absolutely continuous with respect to the uniform measure $\sigma$ and the corresponding result for $\sigma$ is classical, see for instance \cite[Lemma 5.2]{bak}. 

\medskip

To sumarize, we have seen that the first statement in Theorem \ref{main} is a rather straightforward consequence of Theorem \ref{KantoPotenI}. In the rest of the paper, we shall  first prove that the mass transport problem is well-posed then buil a Kantorovitch potential as stated in Theorem \ref{KantoPotenI}. Our approach consists in discretising several probability measures at a scale compatible with the weak Aleksandrov condition.

In the next part we introduce suitable discretisations of the involved measures based on a partition of the sphere into Borel sets of small diameters. We then use this construction first to single out a transport plan with finite cost thus proving the mass transport problem is well-posed. Second, we study a graph induced by the discretised measures before proving Theorem \ref{KantoPotenI} by constructing a suitable Kantorovitch potential in Section 5. Our approach is a generalisation of methods introduced in \cite{BP, BPP}. Finally, we prove the second statement in Theorem \ref{main}.

\section{Discretisation of the measures}\label{section-discrete}


We first recall a basic covering lemma and then draw some consequences of this construction

\begin{lemma}\label{lem-cover}Let $\theta$ be a finite Borel measure on the unit sphere $\S$ endowed with the spherical distance $d$. For any $\kappa>0$, there exists a finite partition $(P_i)_{1\leq i\leq K}$ of $\S$ (depending on $\kappa$) such that for all $i$,  the interior $\stackrel{\circ}{P}_i$ of ${P}_i$ is nonempty, $diam (P_i) < \kappa$ and $\theta (\partial P_i)=0$. If we also assume that $\theta$ is absolutely continuous with respect to the uniform measure on $\S$, we can further require $\theta(P_i)$ to be a rational number.
\end{lemma}
A very closed result is proved in the appendix of \cite{Bertrand-GeomD}. In the same paper, and building on the above covering lemma, it is shown that if the uniform probability measure and $\mu$ are Aleksandrov related then there exists a plan  $\pi^{\alpha} \in \Gamma(\lambda,\mu)$ and $M>0$ such that $c \leq M<+\infty$ everywhere on the support of $\pi^{\alpha}$. As noticed in \cite[Remark 4.9]{Bertrand-GeomD}, the same proof applies to $\lambda$ and $\mu$ as well; moreover the proof precisely requires the \emph{weak} Aleksandrov condition to hold. We refer to the appendix for a proof of the above lemma and the following theorem.

\begin{thm}\label{thm-finite-cost-plan} Let  $\lambda$  and $\mu$ be two probability measures on $\S$ satisfying the weak Aleksandrov condition \eqref{eq-WAA}.  Assume that $\lambda$ is absolutely continuous while the support of $\mu$ is not contained in a closed hemisphere. Then there exists a transport plan $\pi^{\alpha} \in \Gamma(\lambda,\mu)$ such that
$$ \supp \pi^{\alpha} \subset \{(n,x) \in \S \times \S;  \langle n , x \rangle \geq \cos (\pi/2 -\alpha) \}.$$
In particular, 
$$\int c \, d\pi^{\alpha} < +\infty.$$
\end{thm}

\begin{cor}\label{cor-optimalPlan} Under the above assumptions the optimal mass transport problem for the measures $\lambda$ and $\mu$ relative to the cost function $c$ has a solution. Let us denote by $\pi_o$ an optimal transport plan.

\end{cor}
\begin{proof}
This is a very standard consequence of the continuity of the cost function and the compactness of $\Gamma(\lambda,\mu)$, see for instance \cite{Vi03} for a proof.
\end{proof}

With those tools at our disposal, we can now define a graph involving the transport plans $\pi^{\alpha}$ and $\pi_o$.

\smallskip

\noindent {\bf Coverings of $\supp \mu$ and $\supp \nu$}\\

According to Lemma \ref{lem-cover} applied to $\supp \mu$ and $\supp \nu$,  there exist  two finite collections of Borel subsets of $\S$ with nonempty interiors $(P_i)_{1\leq i \leq l} $ and $(Q_j)_{1 \leq j \leq p} $  such that

\begin{equation}\label{eqn-cover-supp}
\lambda (P_i)>0, \mu(Q_j)>0,\;\lambda (\partial P_i)= \mu(\partial Q_j)=0,\; \di (P_i), \di (Q_j) < \alpha/8,  
\end{equation}
for all   $i \in  \{1,\cdots  ,l\}, j \in \{1,\cdots  ,p\}.$


\smallskip

We also fix 
\begin{equation}\label{eq-def-zi-wj}
z_i \mbox{ a point in } \stackrel{\circ}{P_i}\cap \, \supp \mu \;\mbox{ and }\;w_j \mbox{ a point in } \stackrel{\circ}{Q_j} \cap\, \supp \nu.
\end{equation}

 We then introduce the discretised measures $\lambda^d$ and $\mu^d$ relative to the above covers as follows:

$$\lambda^d := \sum_{1\leq i\leq l} \lambda(P_i)\,  \delta_{z_i} \mbox { and } \mu^d := \sum_{1\leq j\leq p} \mu(Q_j)\,  \delta_{w_j}.$$

We proceed similarly for transport plans. We define the following discrete plans

$$\pi_o^d:= \sum_{i,j} \pi_o(P_i\times Q_j)\, \delta_{(z_i,w_j)} \;\; \mbox{ and } \pi^{\alpha, d}:= \sum_{i,j} \pi^{\alpha}(P_i\times Q_j)\, \delta_{(z_i,w_j)}.$$

\medskip

\noindent {\bf Graph structure on the product covering}\\

Let us now define the following \emph{finite graph}

$$ \G:=\{(P_i,Q_j); \pi_o(P_i\times Q_j)>0\} \subset  \{(P_i, Q_j); 1\leq i\leq l, 1\leq j \leq p\}.$$

For simplicity, we shall sometimes write $(i,j) \in \G$ instead of $(P_i,Q_j) \in \G$.

The graph $\G$ is equipped with the following set of \emph{primal oriented edges}
$$\E=\{\{(i,j), (u,v)\}; \pi^{\alpha}(P_u\times Q_j)>0\}.$$

We call \emph{length} of a path in $(\G,\E)$ the number of edges it uses.

\begin{remark}
The graph $\G$ has many edges since, given $\{(i,j), (u,v)\} \in \E$, we get that 
$\{(o,j), (u,s)\} \in \E$ for any $o, s$, such that $(o,j) \in \G$ and $(u,s)\in \G$.

\end{remark}
We are interested in the cycles
 of $\G$, namely the finite collections $(i_1,j_1), \cdots ,(i_q,j_q) \in \G$ where $q$ is an arbitray positive integer, and such that $\{\{(i_s,j_s) ,(i_{s+1}, j_{s+1})\} \} \in \E$ for all $s \in \{1, \cdots, q\}$ (where $(i_1,j_1)=(i_{q+1},j_{q+1})$).

\begin{lemma}\label{lemma-cycle} 
Each edge $e \in \E$ belongs to a least one cycle $(i_1,j_1), \cdots ,(i_q,j_q) \in \G$, say $e = \{(i_1,j_1) ,(i_2 ,j_2)\}$. Moreover, one can assume that the vertices of the cycle are all distinct. 
\end{lemma}

\begin{remark}\label{remark-cycle}
Since any vertex of $\G$ belongs to an edge, the above lemma also implies that any vertex of $\G$ belongs to a cycle.
\end{remark}

\begin{proof}
By definition of an edge, $\pi^{\alpha}( P_{i_2}\times Q_{j_1})>0$ and $\pi_o(P_{i_2}\times Q_{j_2})>0$. We have the following alternative, either $ \pi^{\alpha}((\R^n \setminus P_{i_1})   \times Q_{j_2})=0$ (which implies that $(i_1,j_1), (i_2,j_2)$ is a cycle)  or we can repeat the argument with $\{(i_2,j_2), (i_3,j_3)\}\neq e$ since there exists $P_{i_3}\neq P_{i_1}$ such that $\pi^{\alpha}(P_{i_3}\times  Q_{j_2})>0$ combined with the fact that $\pi_o \in \Gamma(\mu,\nu)$ implies the existence of $ Q_{j_3}$ such that $(P_{i_3}, Q_{j_3}) \in \G$. In each case, after finitely many steps we obtain a cycle; however we cannot infer that $e$ belongs to it.  Thus, let us further study the graph $(\G,\E)$. Let  $(i_1,j_1), \cdots, (i_k,j_k) \in \G$ be an arbitrary cycle made of distinct points and
$$ m:= \min \left(\min_{1\leq u \leq k} \{ \pi_o(P_{i_u}\times Q_{j_u})\}\; , \min_{1\leq u \leq k} \{ \pi^{\alpha}(P_{i_{u+1}}\times Q_{j_u})\}\right)>0.$$

Consider $\pi_1^d=\pi_o^d -\sum_{u=1}^k m  \,  \delta_{(z_{i_u},w_{j_u})}$ and $\pi_1^{\alpha , d}= \pi^{\alpha , d} -\sum_{u=1}^k m \,  \delta_{(z_{i_{u+1}},w_{j_u})}$, note that $\pi_1^d$ and $\pi_1^{\alpha, d}$ have the same marginals. Moreover, by definition of $m$, 
 $$\sharp \,\supp \pi_1^d + \sharp  \, \supp \pi_1^{\alpha , d} \leq \sharp \, \supp \pi_o^d + \sharp \, \supp\pi^{\alpha, d}-1 $$ (where $\sharp$ stands for the cardinal of a set). Now define the reduced graph $\G_1$ similarly to what we did for $\G$ but using $\pi_1^{d}$ and $\pi_1^{\alpha, d}$ instead of $ \pi_o^d$ and $\pi^{\alpha , d}$. By construction, $\G_1 \subset \G$, $\E_1 \subset \E$, and, as explained above, we have $\sharp \, \G_1 + \sharp \, \E_1 \leq \sharp \, \G + \sharp  \, \E -1$. Since $\sharp \G + \sharp \E  \leq lp+ (lp)^2$, after repeating the construction at most $lp+(lp)^2$ times, we get an empty reduced graph. This precisely means that any edge of graph $\G$ belongs to a cycle.

\end{proof}

We call \emph{bunch of cycles} related to $(P_0,Q_0)$, and denote by $\B_{(P_0,Q_0)}$, the set of vertices of $G$ sharing a cycle with $(P_0,Q_0) \in G$, namely the set of vertices $(P,Q)$ such that there exist a cycle $C_y$ (depending on $(P,Q)$) with $(P,Q)$ and $(P_0,Q_0)$ belong to $C_y$. The set of bunches of cycles induces a partition of the graph $\G$, and, in general, more than one bunch of cycles is needed to cover $\G$.

Being in the same bunch of cycles induces an equivalence relation on $\G$. In the next definition, we enlarge the set of edges to be able to connect more vertices at the price of breaking the symmetry of the previous relation.

We now equipped $\G$ with an \emph{enlarged set of oriented edges} $\E'$ defined by
\begin{equation}\label{eq-EnlEdges}
\E'=\{\{(i,j), (u,v)\}; d(z_u, w_j)< \pi/2-\alpha/4\}.
\end{equation}

\begin{remark}
The fact that $\supp \pi^{\alpha} $ is made of pairs of points at distance at most $\pi/2 - \alpha$ together with the upper bound on the diameter of the $(P_i)$'s and $(Q_)j$'s readily imply the inclusion: $\E \subset \E'$. 
\end{remark}

\medskip

\noindent{\bf Chains between points and connected subgraph}\\

Let us start with a weak notion of connectedness.

\begin{defi}[$\kappa$-chainable space]
Let $\kappa>0$ and $(X,d)$ be a metric space. An ordered set $\{x_1, \cdots, x_p\} $ in $X$ such that $d(x_i,x_{i+1}) < \kappa$ for any $i=1, \cdots, p-1$ is said to be a $\kappa$-chain (of length $p$) from $x_1$ to $x_p$. The relation $x \underset{\kappa}{\sim} y$ iff there exists a finite chain  from $x$ to $y$ induces an equivalence relation on $X$. By analogy, we call $\kappa$-chain connected component an equivalence class of $\underset{\kappa}{\sim}$.    
\end{defi}

It is an easy exercise to show that a connected space is  $\kappa$-connected for any $\kappa>0$. For us, the main adavantage of this notion is given by the following simple lemma.

\begin{lemma}
Let $(X,d)$ be a \emph{compact} metric space and $\kappa>0$. Then $X$ can be decomposed into finitely many $\kappa$-chain connected components.

\end{lemma}

\begin{proof}

 Each $\kappa$-chain connected component is both an open and a closed subset of $X$. Thus, the compactness of $X$ implies the existence of a finite subcover. 

\end{proof}

In what follows, we set
\begin{equation}\label{eq-defKappa}
 {\kappa = \alpha/4.}
\end{equation}

In the rest of this part, we let  $C$ be \emph{a $\kappa$-chain connected component of $\supp \nu$}. We also define the subgraph

\begin{equation}\label{eq-Gc}
\G_C :=\{(P,Q) \in \G; Q \cap C\neq \emptyset\}
\end{equation}
equipped with the set of edges coming from $\E'$.

\begin{remark}\label{rem-path-in-Gc}
Of course, there are a priori edges in $\E'$ connecting points not in $\G_C$. In the rest of the paper when we consider \emph{a path in $\E'$ connecting two points in $\G_C$} we do not intend to restrict to edges in $\E'$ connecting points in $\G_C$, on the contrary any edge in $\E'$ can be part of such a path.
\end{remark}

Our aim is now to prove

\begin{prop}\label{prop-Gc-Conn}
The graph $(\G_C,\E')$ is connected, meaning that any two points of $\G_C$ can be connected by a path in $\E'$. We can further assume the path length is at most $\sharp \G$.
\end{prop}

\begin{proof}

Given $(P,Q), (\tilde P, \tilde Q) \in \G$ such that $Q \cap C \neq \emptyset$ and $\tilde Q \cap C \neq \emptyset$ let us show there exists a path  in $\E'$ from $(P,Q)$ to $\tilde P,\tilde Q)$.  We fix $x \in Q\cap C$ and $\tilde x \in \tilde Q \cap C$. By definition of $C$, there exists a $\kappa$-chain $\{x_1, 
\cdots, x_k\} $ from $x=x_1$ to $\tilde x=x_k$. Since the chain is made of points in $\supp \nu$, there exists maps $\sigma$ (resp. $\theta$) from $\{1,\cdots, k\} $ to $\{1,\cdots, p\} $ (resp. $\{1,\cdots, l\} $) such that 
$x_i \in Q_{\sigma(i)}$. Moreover, there exists $P_{\theta(i)}$ such that $(P_{\theta(i)},Q_{\sigma(i)})\in \G_C$ (since $x_i \in C$ by construction) for $i=2,\cdots,k-1$. 
Thanks to Lemma \ref{lemma-cycle}, there exists a cycle $c$ made of edges in $\E$ going through $(\tilde P, \tilde Q)$. Let us denote by $\{(\tilde P,\tilde Q), (\hat P, \hat 
Q)\}$ the edge issuing from $(\tilde P, \tilde Q)$ in this cycle. We denote by $(\hat z, \hat w)$ the pair of representative points of $(\hat P, \hat Q)   $. The 
triangle inequality then gives
\begin{multline*} d(\hat z, w_{\sigma(p-1)}) \leq d(\hat z,a) +d(a,b) +d(b,  \tilde x) + d(\tilde x, x_{p-1}) + d(x_{p-1}, w_{\sigma(p-1)}) \\
 < \alpha/8 + \pi/2 - \alpha + \alpha/8  +  \alpha/4  +  \alpha/8 < \pi/2 - \alpha/4,
 \end{multline*}
where $(a,b) \in \supp \pi^{\alpha}\cap ( \hat P\times \tilde Q )$. The previous estimate precisely means that 
$$\{(P_{\theta(p-1)},Q_{\sigma(p-1)}), (\hat P, \hat Q)\} \in \E'.$$
 Therefore, $ (P_{\theta(p-1)},Q_{\sigma(p-1)})$ is connected to $ (\tilde P, \tilde Q) $ by following part of the cyle $c$ introduced above. The thesis then follows by a finite induction. The last property follows from the fact that without loss of generality one can assume the path contains no cycle.

\end{proof}


\section{Building a Kantorovitch potential}

In this part, building on the results from Section \ref{section-discrete}, we prove the existence of the Kantorovitch potential needed to solve the Gauss image problem. Our proof is based on a Rockafellar-Ruschëndorf formula. Let us recall the result we shall prove:

\begin{thm}\label{KantoPoten}Let $\lambda$ and $\mu$ two probability measures on $\S$  satisfying the weak Aleksandrov condition $\eqref{def-WAA}$ and assume that $\lambda$ is absolutely continuous while the support of $\mu$ is not contained in a closed hemisphere. Then, denoting by $\pi_o$ an optimal plan in $\Gamma (\lambda, \mu)$, there exist a Kantorovitch potential $\varphi$ such that 
$$ \supp \pi_o \cap \{c<+\infty\} \; \subset \partial_c \varphi.$$

\end{thm}

We first prove auxiliary results. As recalled above, the support of any optimal plan relative to $c$ is $c$-cyclically monotone \eqref{eq-def-cmono}. The set 
$$\Gamma:= \supp (\pi_o) \cap \{c<+\infty\}$$
 is then $c$-cyclically monotone as a subset of the support of $\pi_o$.
 
 According to the results in Section \ref{CP&OMT}, 
 $$ \int c \, d\pi_o < +\infty$$
 whenever $\pi_o \in \Gamma_{opt}(\sigma,\mu)$. Therefore, 
 $\pi_o(\{c<+\infty\})=1$
 and $\Gamma$ is a set of full $\pi_o$-measure.
 
 
 \subsection{$\kappa$-chains and $c$-path boundedness}

\begin{defi}\label{$c$-path boundedness}\footnote{This definition differs from the one in \cite{arstein}} A pair $(n,x)\in \Gamma$ is said to be $c$-path connected to $ (\tilde n, \tilde x) \in \Gamma$ if there exists an ordered set $\gamma:=\{(n_1,x_1)\cdots, (n_k,x_k)\}$ in $\Gamma$, called a $c$-path, such that $d(n_{i+1},x_i) <\pi/2$ for $i=1,\cdots k-1$, $(n,x)=(n_1,x_1)$ and $(\tilde n,\tilde x) = (n_k,x_k)$. The \emph{cost} $c(\gamma)$ of $\gamma$ is said to be finite if
$$ c(\gamma) := \sum_{i=1}^{k-1} c(n_{i+1},x_i)- c(n_i,x_i)  \in \R .$$ 
When $c(\gamma) \in \R$ we call $\gamma$ a \emph{bounded} $c$-path from $(n,x)$ to $(\tilde n, \tilde x)$.
\end{defi}

Note that the cost of an arbitrary $c$-path (in $\Gamma$) is in $\R \cup \{+\infty\}$. 
In the next proposition, building upon our study of the graph $(\G,\E')$  and its subgraphs, we get properties on $c$-path boundedness. Recall that $\kappa =\alpha/4$.

\begin{prop}\label{prop-combina} Let $C$ be a $\kappa$-chain connected component of $\supp \nu$. Then for any $(n,x), (\tilde n, \tilde x) \in \Gamma$ such that $x, \tilde x \in C$, the pair $(n,x)$ is $c$-path connected to $(\tilde n,\tilde x)$.  More precisely, there exists a positive constant $C(\alpha)$ such that for any $(n,x), (\tilde n, \tilde x)$ as above there exists a $c$-path $\gamma$ from $(n,x)$ to $(\tilde n, \tilde x)$ whose cost satisfies 
$$ c(\gamma) \leq C(\alpha).$$
\end{prop}

\begin{remark} The above proposition is a generalisation of \cite[Lemma 5.5]{BPP}.
\end{remark}

\begin{proof}
We set $(P,Q) \in \G$ and $(\tilde P, \tilde Q) \in \G$ such that $(n,x) \in P\times Q$ and $(\tilde n, \tilde x) \in \tilde P \times \tilde Q$. According to Proposition \ref{prop-Gc-Conn}, we infer the existence of an ordered set $ \{ (P_{i_1},Q_{j_1}), \cdots, (P_{i_k},Q_{j_k})\}   \in \G$ connecting $(P,Q)$ to $(\tilde P, \tilde Q)$ through edges of $\E'$. We denote by $(z_{i_s}, w_{j_s})$ for $s=1,\cdots , k$ the associated pairs  of representative points.  We claim that the path 
$$ \gamma = \{(n,x), (z_{i_2},w_{j_2}),\cdots , (z_{i_{k-1}},w_{j_{k-1}}), (\tilde n, \tilde x)\}$$
 has finite cost. Indeed, by definition of $\E'$, we have $d(z_{i_{s+1}},  w_{j_s}) < \pi/2 - \alpha/4$ while 
$$ d(z_{i_2}, x) \leq d(z_{i_2},w_{j_1}) + d(w_{j_1}, x) < \pi/2 -\alpha/4 + \alpha/8= \pi/2-\alpha/8,$$
and the same inequality holds for $d(\tilde n, w_{j_{k-1}}).$
As a consequence, we get 
\begin{equation}\label{eq-est-cgam}
c(\gamma) \leq c(z_{i_2},x) +\sum_{s=2}^{k-2} c(z_{i_{s+1}}, w_{j_s}) + c(\tilde n, w_{j_{k-1}}) \leq \sharp \G \cdot C(\pi/2-\alpha/8)<+\infty,
\end{equation}
where $C(\beta)=-\ln (\cos \beta)$, and the proof is complete.
\end{proof}

The next step consists in building a well-behaved function $\psi$ whose $c$-transform is the potential $\phi$ we need to prove our main result. For that purpose, we introduce some extra notation. In what follows, we denote by $p_x :\S \times \S \longrightarrow \S$ the canonical projection on the $x$-variable. For each $\kappa$-chain connected component $C_i$ of $\supp \nu$ such that $C_i \cap p_x(\Gamma) \neq \emptyset$, we fix a pair $ (n^{(i)},x^{(i)}) \in \Gamma$ such that $x^{(i)} \in C_i$. Using these $\kappa$-chain connected components, we can  decompose $p_x(\Gamma)$ as follows

\begin{equation}\label{eq-DecomPxG}
p_x(\Gamma) = \sqcup_{i \in I}    C_i \cap p_x(\Gamma).
\end{equation}
Recall that $I$ is a finite set. With a slight abuse of notation, we shal call $C_i \cap p_x(\Gamma)$ a $\kappa$-chain connected component of $ p_x(\Gamma)$. 

\medskip

The definition of $\psi$ depends on the $\kappa$-chain connected components of $ p_x(\Gamma)$. Namely, $\psi\equiv - \infty$ out of $p_x(\Gamma)$ while the definition of $\psi$ depends on the decomposition \eqref{eq-DecomPxG}. We first define the function $\psi_{C_i}$  on $C_i\cap p_x(\Gamma)$ by the formula

\begin{equation}\label{eq-def-PsiC}
\psi_{C_i}(x)  := \sup_{n, \gamma; (n,x) \in \Gamma }  -c(\gamma) + c(n,x),
\end{equation}
where $\gamma$ is a (bounded) $c$-path from $(n^{(i)},x^{(i)})$ to $(n,x)\in \Gamma$. We have seen in the previous proposition that a bounded $c$-path exists yet the function $\psi_{C_i}$ could be infinite at some points. However we have $ \psi_{C_i} > -\infty$ on $C_i \cap p_x(\Gamma)$. In the next lemma, we prove that $\psi_{C_i}$ is bounded from above.

\begin{lemma}\label{lem-psiC-bound} There exists a positive constant $C=C(\alpha)$ such that

$$ \psi_{C_i} \le C(\alpha)$$
on the $\kappa$-chain connected component $C_i$ of $\supp \nu$. Moreover, $  \psi_{C_i}$ is real-valued on $C_i \cap p_x(\Gamma)$.

\end{lemma} 

\begin{remark}
The constant $ C(\alpha)$ does not depend on $C_i$.

\end{remark}
\begin{proof}
Let $x \in C_i \cap p_x(\Gamma)$ and $n$ such that $(n,x) \in \Gamma$. According to Proposition \ref{prop-combina}, we can  fix a \emph{bounded} $c$ path $\gamma_x$  from $(n,x)$ to the representative pair  $(n^{(i)},x^{(i)})$ of $C_i$. Note that by definition of the cost, $c(\gamma_x)$ does not depend on $n$.  Let $\gamma$ be an  arbitrary $c$-path from  $(n^{(i)},x^{(i)})$ to $(n,x)$. We then estimate

$$
-c(\gamma)  -c(\gamma_x) = \sum_{r=1}^k c(n_r,x_r) - c(n_{r+1},x_r) \leq 0,
$$
where there exists $s $ between $1$ and $k$ such that 
$\gamma= \{(n_1,x_1),\cdots, (n_s,x_s)\}$ and \\ $\gamma_x = \{  (n_s,x_s),\cdots, (n_{k+1},x_{k+1})\}$. By concatenating the two paths, we get a cycle (in the sense of optimal mass transport); the $c$-cyclical monotonicity of $\Gamma$  (see \eqref{eq-def-cmono}) then implies the last inequality.

Using again that $(n_s,x_s)=(n,x)$ we infer

$$ -c(\gamma) + c(n,x) \leq c(n_{s+1},x) + \sum_{r=s+1}^k c(n_{r+1},x_r) -c(n_r,x_r) \leq \sum_{r=s}^k c(n_{i+1},x_i) \leq C(\alpha)<+\infty, $$
thanks to \eqref{eq-est-cgam}. The second statement has been explained prior to the statement.
\end{proof}

In order to define $\psi$, we first need to connect, when possible, distinct $\kappa$-chain connected components. Given $i \neq j \in I$, let us fix $\gamma_{ij}$ a bounded $c$-path from the representative pair $(n^{(i)},x^{(i)})$ to $(n^{(j)},x^{(j)})$ and denote by $c_{ij}:=c(\gamma_{ij})$ the associated cost. When such a path does not exist, we set $c_{ij}:=+\infty$; by convention we define $c_{ii}=0$ for any $i  \in I$. We can now define the function $\psi$ as follows:

\begin{equation}\label{eq-def-psi}
\psi(x) =\left\{ \begin{array}{lcl} -\infty & \mbox{ if } & x \notin p_x(\Gamma) \\
                                    max_{i \in I}   -c_{ij} + \psi_{C_j}(x) & \mbox{ if } & x \in C_j \cap p_x(\Gamma) 
\end{array}\right.
\end{equation}

By definition of the $c_{ij}$'s and according to Lemma \ref{lem-psiC-bound}, there exists a positive constant $C=C(\alpha)$ such that

$$ \psi > -\infty  \mbox{ on } p_x(\Gamma)   \mbox{    and    }  \psi \leq C(\alpha)<+\infty.$$

Recall that by assumption, the measure $\mu$ gives mass to any open hemisphere $B(x,\pi/2)$ thus $p_x(\Gamma)\cap B(x,\pi/2)\neq \emptyset$ for any $x \in \S$ (otherwise $\supp \mu= \overline{p_x(\Gamma)} \subset \overline{B}(-x,\pi/2)$ hence a contradiction). Therefore one can invoke the following proposition from \cite[Proposition B.3]{BeCa}:

\begin{prop}\label{prop-LipRegPo}
Let $\psi :  \S \longrightarrow \R \cup \{-\infty\}$ be a function bounded from above such that
$$\forall u \in \S, \;\; B(u,\pi/2) \cap \{\psi >-\infty\} \neq \emptyset.$$
Then $\psi^c$ is real-valued and Lipschitz regular on $\S$, moreover its Lipschitz constant only depends on upper bounds on $\psi$ and $\psi^c$.

\end{prop}

We finally set
\begin{equation}\label{eq-def-phi}
\phi:= \psi^c.
\end{equation}
Thus $\phi$ is real-valued and Lipschitz regular. Consequently $\phi^c$ and $\phi^{cc}$ share the same properties. Finally, it is a classical result (see for instance \cite{Vi08}) that $\phi^{cc}= \psi^{ccc} = \psi^c=\phi$. In other terms, the function $\phi$ is a Kantorovitch potential. In order to complete the proof of Theorem \ref{KantoPoten}, we are left with proving the following result.

\begin{lemma}\label{lem-gamm-dcphi}
Under the assumptions of Theorem \ref{KantoPoten}, the following inequality holds

$$ \Gamma \subset \partial_c \phi.$$

\end{lemma}

\begin{proof}
Let $(\bar n, \bar x) \in \Gamma$. By definition of the $c$-transform, we are done if we can prove
$$ c(\bar n, \bar x) -\phi(\bar n) \leq c(n,\bar x) -\phi(n),$$
for all $n \in \S$. This inequality is equivalent to
\begin{equation}\label{eq-pract-phi}
\phi(n) \leq \phi(\bar n) + c(n,\bar x) -c(\bar n, \bar x),
\end{equation}
for all $n \in \S$. 
For convenience we set $\mathcal{S}:=\{(n^{(i)},x^{(i)}), i \in  I\}$ the collection of representative pairs of the $\kappa$-chain connected components $\big(C_i\cap p_x(\Gamma)\big)_{i\in I}$.

By combining
$$\phi(n) = \inf_{x \in \S} c(n,x) -\psi(x)$$
together with the expression for $x \in C_j\cap p_x(\Gamma)$ (we can discard the other points since $\phi$ is real-valued according to the previous proposition):
$$ \psi(x) =\sup_{i \in I, k \in \N} \sup_{(n_s,x_s)\in \Gamma^k,\, (n_0,x_0) = (n^{(i)},x^{(i)}), x_k=x}  -c_{ij}+ \left(\sum_{s=0}^{k-1} c(n_s,x_s) - c(n_{s+1}, x_s)\right)  \,\, + c(n_k,x),$$

where $ j$ is defined by $x \in C_j$, and there is no second term in the right hand side when $k=0$.

Recall that $c_{ii}=0$ and $c_{ij}=+\infty$ if there is no bounded $c$-path from $(n^{(i)},x^{(i)})$ to $(n^{(j)},x^{(j)})$. Therefore, one can discard these $i$'s in the above definition since for $i=j$ and $x \in  C_j \cap p_x(\Gamma)$, $\psi_{C_j}(x) \in \R$.

Consequently, we can write the function $\phi$ in a similar fashion as $\psi$, namely

$$ \phi (n) =\inf_{i \in I, k \in \N} \;\;\inf_{(n_s,x_s)\in \Gamma^k,\, (n_0,x_0) = (n^{(i)},x^{(i)}),n_{k+1}=n} c_{ij} + \sum_{s=0}^{k} c(n_{s+1}, x_s) -c(n_s,x_s).$$
 
 For $(n_0,x_0)$ for which $c_{ij}$ is finite (and $i\neq j$), the expression in the RHS above can be written as $c(\tilde \gamma)$ where $\tilde \gamma$ is the concatenation of the $c$-path $\gamma_{ij}$ together with the $c$-path from $(n_0,x_0)$ to $(n,x)$ described in the formula. Consequently, for $n=\bar n$ if we further add to $\tilde \gamma$ above the $c$-path $\{(\bar n, \bar x), (n,x)\}$, one easily get the expected formula \eqref{eq-pract-phi} by considering a minimizing sequence of $c$-paths relative to $\phi(\bar n)$.
\end{proof}

\subsection{On a first order uniqueness of the solution}

In this part, we prove:

\begin{thm}\label{th-uniqueness}
Let $\lambda$ and $\mu$ two probability measures on $\S$,   and assume that $\lambda$ is absolutely continuous. Assume there are two convex bodies $K,L \in \K_0$ solutions to \eqref{eq-GIPro}. Then, the maps $\mathcal{G}_k\circ \vec{\rho}_K$ and $ \mathcal{G}_L\circ \vec{\rho}_L$ coincide $\lambda$-a.e. as multivalued maps, namely for any Borel set $\omega$

$$ \lambda ((\mathcal{G}_k\circ \vec{\rho}_K(\omega)) \Delta \, (\mathcal{G}_L\circ \vec{\rho}_L(\omega)))=0.$$

\end{thm}

\begin{remark}
Recall that for a convex in $\K_0$, the fact that the origin belongs to \emph{the interior} of the convex body prevents the angle between a direction $x$ and a normal vector $n \in \mathcal{G} \circ \vec{\rho} (x)$ from being too close to $\pi/2$. In other terms, the measures $\lambda$ and $\mu$ \emph{must} satisfy the weak Aleksandrov condition for a sufficiently small $\alpha>0$. Consequently, as proved in Theorem \ref{thm-finite-cost-plan}, the mass transport problem relative to $\mu, \lambda$ and the cost $c$ is well-posed. We keep the notation $\pi_o$ for the optimal plan.

\end{remark}

\begin{proof}
Observe that a solution $(\rho,h) \in \{(\rho_K,h_K),(\rho_L,h_L)\}$ to the Gauss image problem becomes, after applying Oliker's change of functions, a solution to the Kantorovitch problem; we denote by  $(\phi_K,\psi_K)$ and $(\phi_L, \psi_L)$ these solutions. Indeed, given a convex body $K \in \K_0$,  up to a Lebesgue negligible  set $N_K$, for all $n \in \S\setminus N_K$ there exists a unique $x \in \S$ such that $n \in \mathcal{G}_K (\vec{\rho}_K(x))$. Besides if we denote by $T_K(n)$ such a $x$ then it is known that $T_K$ is continuous on $\S \setminus N_K$ \cite{Rocka70}. Thus for any Borel set $\omega \subset \S$,
$$\mu(\omega) =\lambda ( \mathcal{G}_K \circ \vec{\rho}_K(\omega)) = \lambda(T_K^{-1} (\omega)),$$
in other terms the pushforward of $\lambda$ through $T_K$ is $\mu$. This property is denoted by ${T_K}_{\sharp} \lambda = \mu$.

Reasoning as in the proof of Theorem \ref{KantoPotenI}, we get for $(\phi_K,\psi_K)$:

\begin{eqnarray*}
\int_{\S \times \S} c(n,x) \, d(Id,T_K)_{\sharp}\lambda (n) &=&   \int_{\S} c(n,T_K(n)) \, d\lambda (n) \\
&= & \int_{\S}\big( \phi_K(n)+ \psi_K(T_K(n))\big) \, d\lambda (n) \\
&=&   \int_{\S} \phi_K(n) \, d\lambda (n) + \int_K \psi_K(T_K(n)) \, d\lambda (n) \\
&=&  \int_{\S} \phi_K(n) \, d\lambda (n) + \int_K \psi_K(x) \, d\mu (x). 
\end{eqnarray*}

Therefore $(\phi_K,\psi_K)$ is a solution to the dual problem (and $(Id,T_K)_{\sharp}\lambda $ is an optimal plan). The same properties hold for the convex body $L$ and the corresponding objects.

Consequently, 
\begin{equation}\label{eq-triv}
\Gamma = \supp \pi_o \cap \{c<+\infty\} \subset \partial_c\phi_K \cap \partial_c \phi_L.
\end{equation}

Recall that $\pi_o (\Gamma)=1$ since the mass transport problem is well-posed. Besides, for $\omega$ a Borel set in $\S$, observe that

$$ \partial_c \phi_K\cap (\S \times \omega) = \{(n,x); n \in \mathcal{G}_K \circ \vec{\rho}_K (x), x \in \omega\},$$

thus 
$$p_n(\partial_c \phi_K \cap (\S \times \omega)) =  \mathcal{G}_K \circ \vec{\rho}_K(\omega)$$
and the same property holds for $L$ instead of $K$. Now, according to \eqref{eq-triv}
$$\pi_o(\partial_c \phi_K  \cap \partial_c \phi_L\cap (\S \times \omega)) = \pi_o (\S \times \omega)= \mu(\omega).$$

Finally

$$ \lambda ( \mathcal{G}_K \circ \vec{\rho}_K(\omega) \cap  \mathcal{G}_L \circ \vec{\rho}_L(\omega)) \geq \pi_0 (p_n^{-1} (p_n(\partial_c \phi_K  \cap \partial_c \phi_L\cap (\S \times \omega)))) \geq \mu(\omega).$$
Since $K$ and $L$ are solutions to the Gauss image problem, equality actually holds in the above inequality:
$$
\lambda ( \mathcal{G}_K \circ \vec{\rho}_K(\omega) \cap  \mathcal{G}_L \circ \vec{\rho}_L(\omega)) = \lambda ( \mathcal{G}_K \circ \vec{\rho}_K(\omega))=\lambda (  \mathcal{G}_L \circ \vec{\rho}_L(\omega))$$
and the result is proved.

\end{proof}

\section*{Appendix}

In this appendix, we first prove the following result which is a minor adaptation of the appendix of \cite{Bertrand-GeomD}. 

\begin{lemma}\label{parti} Let $\theta$ be a  Borel probability measure on the unit sphere $\S$ endowed with the spherical distance $d$. For any $\kappa>0$, there exists a finite partition $(P_i)_{1\leq i\leq K}$ of $\S$ (depending on $\kappa$) such that for all $i$,  the interior $\stackrel{\circ}{P}_i$ of ${P}_i$ is nonempty, $diam (P_i) < \kappa$ and $\theta (\partial P_i)=0$. If we also assume that $\theta$ is absolutely continuous with respect to the uniform measure on $\S$, we can further require $\theta(P_i)$ to be a rational number.
\end{lemma}

\begin{proof}

The proof is by induction on the dimension $m$. Let us recall the expression of the spherical distance $d$ in spherical coordinates say $(t,u)$ where $t \in [0,\pi]$ and $u \in \mathbb{S}^{m-1}$:

\begin{equation}\label{eq-sdis}
\cos d((t,u),(s,v))= \cos s\, \cos t + \sin s \, \sin t \, \cos d(u,v),
\end{equation}
where $d(u,v)$ is the spherical distance between $u$ and $v$ in $\mathbb{S}^{m-1}$. We also set $p_t$ (resp $p_u$) the projections associated to these coordinates on $(0,\pi)\times \mathbb{S}^{m-1}$.

For $m=1$, fix a number $\alpha_1>0$. Then, partition $\mathbb{S}^1$ into finitely many left-open, right-closed segments $(I_j)_{1\leq j\leq K_1}$ whose length $l(I_j)$ satisfies $ l(I_j)<\alpha_1$; up to slightly moving the intervals -since the condition on the diameter is open- we can further require that $\theta (\partial I_j)=0$ since $\theta$ has at most countably many atoms. When $\theta$ is absolutely continuous, again one can slightly move the boundary of the intervals to make sure that $\theta(I_j) \in \Q$ for $j=1, \cdots, K_1-1$ while preserving the other properties; $\theta(I_{K_1}) \in \Q$ follows from  $\theta(\mathbb{S}^1)=1$.

For $m=2$, fix a point $N\in \mathbb{S}^2$ and $\alpha_2>0$. Consider a partition $(C_i)_{1 \leq i \leq K_2}$ where $C_1$ is the closed spherical cap centered at $N$ with radius $R_1$, $C_i= \{z\in \mathbb{S}^2; R_{i} <  d(N,z) \leq R_{i+1}\}$ for $ i \in \{2,\cdots,K_2-1\}$ and $C_{K_2}$ is the open ball with radius $\pi-R_{K_2}$ and center $-N$. We require that the $(R_i)$'s satisfy: 
$$ \alpha_2/2 < R_1< \alpha_2, \quad    \alpha_2/2 < R_i-R_{i-1} < \alpha_2, \quad  \pi-R_{K_2} < \alpha_2.$$
 Since the atoms of  the measures $(p_t)_{\sharp} (\theta)$ correspond to the radius $r \in (0,\pi)$ for which the sphere $S(N,r):=\{u \in \S; d(u,N) =r\}$ has positive $\theta$-mass, we can further choose the radii $R_i$'s so that $ \theta (S(N,R_i))=0$ for $ i \in \{1,\cdots,K_2\}$. Similarly, we can assume that $\theta (C_i)\in \Q$ when $\theta$ is absolutely continuous. Now, applying the case $m=1$ to each measure $((p_u)_{\sharp} (\theta \res C_i))_{1 \leq i\leq K_2})$, we get a partition $(P_s)_{1\leq s\leq K}$ of $\mathbb{S}^2$ (namely whose elements are of the form $(C_i \cap (p_u)^{-1}(I_j^i))_{i,j}$, where $(I_j^i)_j$ is the partition of $\mathbb{S}^1$ corresponding to $(p_u)_{\sharp} (\theta \res C_i)$). The $(P_s)$'s have nonempty interiors by construction. In addition to that, we recall that the measures $((p_u)_{\sharp} (\theta \res C_i))_{1 \leq i\leq K_2})$ on $\mathbb{S}^1$ are absolutely continuous with respect to the uniform measure on $\mathbb{S}^{1}$ whenever $\theta$ is so. Using that
$$ \theta (P_s)  = \left(\int_{R_{i}}^{R_{i+1}} \sin r \, dr \right)\cdot  (p_u)_{\sharp} (\theta \res C_i)(I^i_j),$$
when $P_s= C_i \cap (p_u)^{-1}(I_j^i)$; we can further assume  $\theta(P_s) \in \Q$ whenever $\theta$ is a.c.. Finally, the expression of the spherical distance \eqref{eq-sdis} implies that the diameter of any $P_s$ is smaller than $\kappa$ provided $\alpha_1$ and $\alpha_2$ are chosen sufficiently small.

The higher dimensional case easily follows from the arguments used for $m=2$.
\end{proof}

Building on the previous lemma, we can prove the mass transport problem \eqref{massdefi} is well-posed. The proof is a straightforward adaptation of \cite[Proof of Theorem 4.1]{Bertrand-GeomD}.

\begin{thm}Let $\lambda$ and $ \mu$ be two probability measures on $\S$ satisfying the assumptions of Theorem \ref{KantoPoten}. There exists a plan $\pi^{\alpha} \in \Gamma(\lambda, \mu) $ such that
$$ \supp \pi^{\alpha}\subset \{(n,x) \in \S\times \S; d(n,x) \leq \pi/2 -\alpha\}.$$
\end{thm}

\begin{proof}
According to the weak Alexsandrov condition, there exists a number $\alpha >0$ such that
\begin{equation}\label{eq-local-for}
 \mu(F) \leq \lambda (F_{\pi/2-2\alpha}),
\end{equation} 
for any closed set $F$ contained in a closed hemisphere (\ref{eq-WAA}). 

The first step of the proof is to show that we can approximate $\mu$ by a finitely supported measure  that still satisfies the above condition up to sligthly decreasing $\alpha$. To this end, we first approximate $
\mu$ by $(\mu * \rho_{\ep})_{\ep < \alpha/4}$, $\rho_{\ep}$ being a family of standard radial mollifiers on $\S$. We fix such an $\ep$ and set $ \hat{\mu}=\mu * \rho_{\ep}$;  by definition, $\hat{\mu}$ is absolutely continuous 
with respect to the uniform measure and satisfies (\ref{eq-local-for}) with $\frac{7\alpha}{4}$ instead of $2\alpha$. The next step is to use a fit partition for $\hat{\mu}$  as in Lemma \ref{parti}: there exists a finite partition 
$(U_i)_{i \in \{1, \cdots N\}}$ of $\S$ made of Borel sets with nonempty interiors such that  
\begin{equation}\label{blorp}
diam (U_i) < \alpha/4 \mbox{ and }   \hat{\mu}(U_i) \in \Q.
\end{equation}

For each $U_i$, choose $x_i \in U_i$ and set 
$$ \mu_e= \sum_{i=1}^{N} \hat{\mu}(U_i)\delta_{ x_i},$$

By assumption on the diameter of $U_i$, $\mu_e$ satisfies  for all closed set $F$ contained in a closed hemisphere:
 \begin{equation}\label{me33} 
  \mu_e (F) \leq \lambda \left(\cup_{x \in F} B\left(x, \pi/2-\frac{3\alpha}{2}\right)\right)
\end{equation}
and the proof of the first step is complete. According to (\ref{blorp}), $\mu_e$ can be rewritten (up to repeating some the $\tilde x_i$'s)
$$ \mu_e= \frac{1}{M} \sum_{i=1}^M \delta_{\tilde x_i}$$
with $M \in \N\setminus\{0\}$ and $\{\tilde x_1,\cdots,  \tilde x_M\}=\{x_1,\cdots,x_N\}$.
.

The next step is to show the existence of $\pi_e \in \Gamma (\lambda, \mu_e)$ such that
 $$ \int_{\S \times \S} c(n,x) \,d\pi_e(n,x) \leq -\ln \left(\sin \left({\alpha}\right)\right).$$
 
To this aim, we now apply Lemma \ref{parti} to the measure $\lambda$ which is absolutely continuous with respect to the uniform measure. The same argument as the one applied to $\hat{\mu}$ leads to the existence of $M'$ and a partition  $(V_s)_{s=1}^{r}$ of $\S$ which satisfies \eqref{blorp}. We first decompose $\lambda$ as follows

$$ \lambda = \sum_{s=1}^r \lambda(V_s)\, {\lambda_s},$$
where  ${\lambda_s} = \frac{1}{\lambda(V_s)} \lambda\res{V_s}$. Using that $\lambda(V_s) \in \Q$ we can proceed as we did for $\mu_e$ and, repeating some of the $ {\lambda_s}$'s if necessary,  rewrite this equality as

$$ \lambda = \sum_{j=1}^{M'} \frac{1}{M'} \tilde{\lambda_j},$$

where $\tilde{\lambda_j} = \frac{1}{\lambda(\tilde V_j)} \lambda\res{\tilde V_j}$ and $\{\tilde V_1,\cdots,  \tilde V_{M'}\}=\{V_1,\cdots,V_r\}$. Note that the above equlity implies for any $s=1,\cdots,r$
 $$\lambda(V_s) = \frac{\sharp \{j; \tilde V_j=V_s\}}{M'}.$$
Now, up to replacing $M$ and $M'$ by the product $M M'$ and repeating the $(\tilde x_i)$'s and the $(\tilde V_j)$'s, we can assume that $M=M'$. Thus we have a collection of sets $(\tilde V_i)_{i=1}^{M}$ and $\{\tilde x _i\} _{i=1}^{M}$ of $\S$ such that $\di( \tilde V_i) \leq \frac{\alpha}{4}$ and $\lambda(V_s)= \frac{\sharp \{j; \tilde V_j=V_s\}}{M}$ for $s=1,\cdots,r$. Finally, we claim that the set-valued map 
 $$
 \begin{array}{rccc}
 F: &\{1,\cdots,M\} &\longrightarrow &\left\{\tilde V_j, j\in \{1, \cdots, M\}\right\} \\
       &        i & \longmapsto & \{\tilde V_j ; \tilde V_j \subset B(\tilde x_i, \pi/2- {\alpha})\}
  \end{array}
  $$
 satisfies the assumptions of the Marriage lemma. Indeed, consider $I$ a subset of $\{1,\cdots,M\}$. Thanks to (\ref{me33}), we have
 $$ \frac{\sharp I }{M} \leq \mu_e (\{\tilde x_i, i \in I\}) \leq \lambda \Big(\cup_{i \in I}B\Big(\tilde x_i, \pi/2-\frac{3\alpha}{2}\Big) \Big).$$
Now, by assumption on the $V_s$'s, we get
 $$ \bigcup_{i \in I}  B\Big(\tilde x_i, \pi/2-\frac{3\alpha}{2}\Big)  \subset \bigcup_{  V_s; V_s \subset B(\tilde x_i, \pi/2 -\alpha)} V_s.$$  
 
 Therefore  $$\mu_e (\{\tilde x_i, i \in I\}) \leq \sum_{ V_s \subset B(\tilde x_i, \pi/2 -\alpha)} \lambda(V_s) =\sum_{ V_s \subset B(\tilde x_i, \pi/2 -\alpha)} \frac{\sharp \{j; \tilde V_j=V_s\}}{M}= \frac{\sharp \{j; \tilde V_j \in F(I)\}}{M},$$
 and the assumptions of the Marriage lemma are satisfied. Consequently, there exists a one-to-one map $f:   \{1,\cdots,M\} \longrightarrow \{\tilde V_i ; i \in \{1, \cdots ,M\}\}$ such that for all $i$, $f(i) \subset B(\tilde x_i, \pi/2- {\alpha})$. This fact clearly entails that the plan which maps the mass $1/M$ located at $\tilde x_i$ uniformly on $f(i)$ is a plan $\pi_e$ in $\Gamma(\lambda,\mu_e)$ such that
\begin{equation}\label{me35} 
\pi_e \left(\left\{(n,x) \in (\S)^2; d(n,x) \leq \pi/2- {\alpha}\right\}\right)=1.
\end{equation}
Note that the bound does not depend on $M$ nor on $\ep$. Therefore, by letting $\ep$ go to $0$, we can construct by the same method a sequence of empirical measures which converges to $\mu$, all of whose elements satisfy (\ref{me35}). Then using the Banach-Alaoglu theorem, we can extract a subsequence of plans which converges to an element of $\Gamma(\lambda,\mu)$ that satisfies (\ref{me35}).
\end{proof}


\bibliographystyle{plain}
\bibliography{bert-bi}

\begin{thebibliography}{10}

\bibitem{AlexCP}
Aleksandr~D. Aleksandrov.
\newblock {\em Convex polyhedra}.
\newblock Springer Monographs in Mathematics. Springer-Verlag, Berlin, 2005.
\newblock Translated from the 1950 Russian edition by N. S. Dairbekov, S. S.
  Kutateladze and A. B. Sossinsky, With comments and bibliography by V. A.
  Zalgaller and appendices by L. A. Shor and Yu. A. Volkov.

\bibitem{arstein}
Shiri. Artstein-Avidan, Shay Sadovsky, and Kasia Wyczesany.
\newblock A {Rockafellar}-type theorem for non-traditional costs.
\newblock {\em Adv. Math.}, 395:25, 2022.
\newblock Id/No 108157.

\bibitem{bak}
Ilya~J. Bakelman.
\newblock {\em Convex analysis and nonlinear geometric elliptic equations}.
\newblock Springer-Verlag, Berlin, 1994.
\newblock With an obituary for the author by William Rundell, Edited by Steven
  D. Taliaferro.

\bibitem{Bertrand-GeomD}
J{\'e}r{\^o}me Bertrand.
\newblock Prescription of {Gauss} curvature using optimal mass transport.
\newblock {\em Geom. Dedicata}, 183:81--99, 2016.

\bibitem{BeCa}
J{\'e}r{\^o}me Bertrand and Philippe Castillon.
\newblock Prescribing the {Gauss} curvature of convex bodies in hyperbolic
  space.
\newblock {\em To appear in CVPDE, Arxiv: arxiv.org/abs/1903.06502}, 2019.

\bibitem{BPP}
J{\'e}r{\^o}me Bertrand, Aldo Pratelli, and Marjolaine Puel.
\newblock Kantorovich potentials and continuity of total cost for relativistic
  cost functions.
\newblock {\em J. Math. Pures Appl. (9)}, 110:93--122, 2018.

\bibitem{BP}
J{\'e}r{\^o}me Bertrand and Marjolaine Puel.
\newblock The optimal mass transport problem for relativistic costs.
\newblock {\em Calc. Var. Partial Differ. Equ.}, 46(1-2):353--374, 2013.

\bibitem{Boroczky-20}
K{\'a}roly~J. B{\"o}r{\"o}czky, Erwin Lutwak, Deane Yang, Gaoyong Zhang, and
  Yiming Zhao.
\newblock The {Gauss} image problem.
\newblock {\em Commun. Pure Appl. Math.}, 73(7):1406--1452, 2020.

\bibitem{huang}
Yong Huang, Erwin Lutwak, Deane Yang, and Gaoyong Zhang.
\newblock Geometric measures in the dual {Brunn}-{Minkowski} theory and their
  associated {Minkowski} problems.
\newblock {\em Acta Math.}, 216(2):325--388, 2016.

\bibitem{Oliker}
Vladimir Oliker.
\newblock Embedding {$\bold S\sp n$} into {$\bold R\sp {n+1}$} with given
  integral {G}auss curvature and optimal mass transport on {$\bold S\sp n$}.
\newblock {\em Adv. Math.}, 213(2):600--620, 2007.

\bibitem{Rocka70}
R.~Tyrrell Rockafellar.
\newblock {\em Convex analysis}.
\newblock Princeton Mathematical Series, No. 28. Princeton University Press,
  Princeton, N.J., 1970.

\bibitem{schneider}
Rolf Schneider.
\newblock {\em Convex bodies: the {B}runn-{M}inkowski theory}, volume~44 of
  {\em Encyclopedia of Mathematics and its Applications}.
\newblock Cambridge University Press, Cambridge, 1993.

\bibitem{Semenov-exi}
Vadim Semenov.
\newblock The {Gauss} image problem with weak {Aleksandrov} condition.
\newblock {\em Arxiv: arxiv.org/abs/2210.16974}, 2022.

\bibitem{Semenov-uni}
Vadim Semenov.
\newblock The uniqueness of the {Gauss} image measure.
\newblock {\em Arxiv: arxiv.org/abs/2305.01779}, 2023.

\bibitem{Vi03}
C{\'e}dric Villani.
\newblock {\em Topics in optimal transportation}, volume~58 of {\em Graduate
  Studies in Mathematics}.
\newblock American Mathematical Society, Providence, RI, 2003.

\bibitem{Vi08}
C{\'e}dric Villani.
\newblock {\em Optimal transport}, volume 338 of {\em Grundlehren der
  Mathematischen Wissenschaften [Fundamental Principles of Mathematical
  Sciences]}.
\newblock Springer-Verlag, Berlin, 2009.
\newblock Old and new.

\end{thebibliography}


\end{document}